\documentclass{amsart}
\usepackage{amssymb}
\usepackage{amsfonts}

\newtheorem{theorem}{Theorem}[section]
\theoremstyle{plain}

\newtheorem{conjecture}[theorem]{Conjecture}
\newtheorem{corollary}[theorem]{Corollary}

\newtheorem{definition}[theorem]{Definition}
\newtheorem{example}[theorem]{Example}

\newtheorem{lemma}[theorem]{Lemma}
\newtheorem{notation}[theorem]{Notation}

\newtheorem{proposition}[theorem]{Proposition}
\newtheorem{remark}[theorem]{Remark}

\numberwithin{equation}{section}
\input{tcilatex}

\begin{document}
\title[Convolutions of orbital measures]{The absolute continuity of
convolutions of orbital measures in $SO(2n+1)$}
\author{Kathryn E. Hare}
\address{Dept. of Pure Mathematics, University of Waterloo\\
Waterloo, Ont., N2L 3G1, Canada}
\email{kehare@uwaterloo.ca}
\thanks{This research is supported in part by NSERC\ 2016-03719}
\thanks{This paper is in final form and no version of it will be submitted
for publication elsewhere.}

\begin{abstract}
Let $G$ be a compact Lie group of Lie type $B_{n},$ such as $SO(2n+1)$. We
characterize the tuples\ $(x_{1},...,x_{L})$ of the elements $x_{j}\in G$
which have the property that the product of their conjugacy classes has
non-empty interior. Equivalently, the convolution product of the orbital
measures supported on their conjugacy classes is absolutely continuous with
respect to Haar measure. The characterization depends on the dimensions of
the largest eigenspaces of each $x_{j}$. Such a characterization was
previously only known for the compact Lie groups of type $A_{n}$.
\end{abstract}

\maketitle

\section{Introduction}

Let $G$ be a compact, connected Lie group with Lie algebra $\mathfrak{g}$.
Given $x\in G$ (or $X\in \mathfrak{g}$) we let $\mu _{x}$ (or $\nu _{X})$
denote the $G$-invariant measure on $G$ (or $\mathfrak{g}$) supported on the
conjugacy class containing $x$ (or adjoint orbit containing $X$). These are
known as orbital measures. It is a classical result due to Ragozin \cite{Ra}
that if the dimension of $G$ non-trivial orbital measures are convolved
together, the resulting measure is absolutely continuous with respect to
Haar measure on $G$ in the first case and with respect to Lebesgue measure
on $\mathfrak{g}$ in the second. Ragozin proved this by using geometric
properties to deduce that the product of dimension of $G$ non-trivial
conjugacy classes (or the sum of $\dim G$ non-trivial adjoint orbits) had
non-empty interior, an equivalent property.

In a series of papers (see \cite{GAFA}-\cite{Adv}, \cite{HJSY}) the author
with various coauthors used tools from harmonic analysis and representation
theory to improve upon Ragozin's result, determining for each $x\in G$ (or $%
X\in \mathfrak{g}$) the minimal integer $L$ such that the convolution of $%
\mu _{x}$ (or $\nu _{X})$ with itself $L$ times was absolutely continuous.
This number $L,$ which depends on $x$ or $X,$ never exceeds $2rankG$. In 
\cite{Wr}, Wright used geometric arguments to extend these results in the
special case of the classical Lie groups and algebras of type $A_{n}$ and
characterized the absolute continuity of $\mu _{x_{1}}\ast \cdot \cdot \cdot
\ast \mu _{x_{L}}$ or $\nu _{X_{1}}\ast \cdot \cdot \cdot \ast \nu _{X_{L}}$
in terms of the dimensions of the largest eigenspaces of the elements $x_{j}$
and $X_{j}$.

Inspired by the algebraic methods that Gracyzk and Sawyer used to study
related problems in the symmetric space setting (see \cite{GS1}-\cite{GS3}),
the author with Gupta in \cite{CJM} almost completely characterized the $L$
tuples $(X_{1},...,X_{L})$ such that $\nu _{X_{1}}\ast \cdot \cdot \cdot
\ast \nu _{X_{L}}$ is absolutely continuous for each of the compact
classical Lie algebras. The characterization can again be expressed in terms
of the dimensions of the largest eigenspaces. As a corollary of this work,
we were also able to similarly characterize the absolute continuity of $\mu
_{x_{1}}\ast \cdot \cdot \cdot \ast \mu _{x_{L}}$ in the special case that
all the $x_{j}$ had the property that $\dim C_{x_{j}}=\dim O_{X_{j}}$ where $%
\exp (X_{j})=x_{j}$. However, in all the compact Lie groups, except those of
type $A_{n},$ there are many elements $x$ which do not have this special
property. For instance, any element $x\in SO(2n+1)$ with $-1$ an eigenvalue
of multiplicity at least four fails to have this special property.

In this note, we adapt the strategy of \cite{CJM} to complete the
characterization of the absolutely continuous convolution products of
orbital measures on the compact Lie groups of type $B_{n}$, the classical
model being the group $SO(2n+1)$. The characterization is more complicated
than the Lie algebra case, but again can be expressed in terms of the
dimensions of the largest eigenspace of each $x_{j}$. We refer the reader to
Theorem \ref{main} and Definition \ref{eligible} for the precise statement.
Our proof depends heavily upon the Lie theory of roots and root vectors, and
particularly the structure of subsystems of annihilating roots of the
elements $x_{j}$. Type $B_{n}$ is unusual in this regard as its root system
has irreducible subsystems of type $D_{j},$ in addition to $B_{j}$ and $%
A_{j} $, in contrast to the situation in the Lie groups of types $C_{n}$, $%
D_{n}$ or even the Lie algebra of type $B_{n}$. It is this feature which
seems to make the characterization so complex.

We conclude with a conjecture for the characterization of absolute
continuity for the classical Lie algebras of type $C_{n}$ and $D_{n}$ and
briefly discuss an approach to its proof.

\section{Notation and Terminology}

\subsection{Basic notation}

For the remainder of this paper, $G_{n}$ will be a compact, connected simple
Lie group of Lie type $B_{n},$ with Lie algebra $\mathfrak{g}_{n}$ and
centre $Z(G_{n})$. At times, we may surpress the subscript $n,$ which is its
rank. We denote by $T_{n}$ a maximal torus of $G_{n}$ and let $\mathfrak{t}%
_{n}$ denote the Lie algebra of $T_{n}$, a maximal torus of $\mathfrak{g}%
_{n} $. The Haar measure on $G$ will be denoted by $m_{G}$. Our model of
these groups will be the matrix group $SO(2n+1),$ the group of $(2n+1)\times
(2n+1) $ real, unitary matrices of determinant $1$.

We take as its torus the block diagonal matrices, with $n$ $2\times 2$ block
matrices of the form 
\begin{equation*}
\left[ 
\begin{array}{cc}
\cos \theta _{j} & -\sin \theta _{j} \\ 
\sin \theta _{j} & \cos \theta _{j}%
\end{array}%
\right] ,
\end{equation*}%
and a $1$ in the final $1\times 1$ block, and identify such a torus element
with the $n$-vector $(e^{i\theta _{1}},...,e^{i\theta _{n}}).$ The Lie
algebra of this Lie group is $so(2n+1),$ the real $(2n+1)\times (2n+1)$
skew-Hermitian matrices and its torus can similarly be identified with
vectors in $\mathbb{R}^{n}$.

We write $[\cdot ,\cdot ]$ for the Lie bracket action and $ad$ for the map: $%
\mathfrak{g}_{n}\rightarrow \mathfrak{g}_{n}$ given by $ad(X)(Y)=[X,Y]$. The
group acts on its Lie algebra by the adjoint action, denoted $Ad(\cdot ),$
and the exponential map, $\exp ,$ is a surjection from $\mathfrak{g}_{n}$
onto $G_{n}$. We recall that 
\begin{equation*}
Ad(\exp M)=\exp (ad(M))=Id+\sum_{k=1}^{\infty }\frac{ad^{k}(M)}{k!},
\end{equation*}%
where $ad^{k}(M)$ is the $k$-fold composition of $ad(M)$.

\subsection{Orbital measures}

Every element in $G_{n}$ is conjugate to some element $x\in T_{n}$ that has
the vector form

\begin{equation}
x=(\underbrace{1,...,1}_{u},\underbrace{-1,...,-1}_{v},\underbrace{%
e^{ia_{1}},...,e^{ia_{1}}}_{s_{1}},...,\underbrace{e^{ia_{m}},...,e^{ia_{m}}}%
_{s_{m}})  \label{x}
\end{equation}%
where the $0<a_{j}<\pi $ are distinct and $u+v+s_{1}+\cdot \cdot \cdot
+s_{m}=n$. Thus $x$ has $1$ as an eigenvalue with multiplicity $2u+1,$ $-1$
as an eigenvalue with multiplicity $2v$ and each $e^{\pm ia_{j}}$ as an
eigenvalue with multiplicity $s_{j}$. We denote the conjugacy class
containing $x$ by 
\begin{equation*}
C_{x}=\{gxg^{-1}:g\in G_{n}\}.
\end{equation*}%
Similarly, every element of $\mathfrak{g}_{n}$ is in the adjoint orbit $%
O_{X}=\{Ad(g)X:g\in G_{n}\}$ of some $X$ $\in \mathfrak{t}_{n}$. Note that
the element $x$ of (\ref{x}) is equal to $\exp X_{x}$ where 
\begin{equation*}
X_{x}=(0,...,0,\pi ,...,\pi ,a_{1},...,a_{1},...,a_{m},...,a_{m})\in 
\mathfrak{t}_{n}
\end{equation*}%
and that $C_{x}\subseteq \exp (O_{X_{x}})$. This inclusion can be proper.

\begin{definition}
(i) For any $x\in T,$ the \textbf{orbital measure} $\mu _{x}$ is the $G$%
-invariant probability measure supported on $C_{x},$ defined by 
\begin{equation*}
\int_{G}fd\mu _{x}=\int_{G}f(gxg^{-1})dm_{G}(g)\text{ for }f\text{
continuous on }G.
\end{equation*}

(ii) For any $X\in \mathfrak{t}$, the \textbf{orbital measure} $\nu _{X}$ is
the $Ad(G)$-invariant probability measure supported on $O_{X},$ defined by%
\begin{equation*}
\int_{\mathfrak{g}}fd\nu _{X}=\int_{G}f(Ad(g)X)dm_{G}(g)\text{ for }f\text{
bounded and continuous on }\mathfrak{g}\text{.}
\end{equation*}
\end{definition}

\begin{definition}
We will say that the tuple $(x_{1},...,x_{L})\in T^{L}$ (or $%
(X_{1},...,X_{L})$ $\in \mathfrak{t}^{L}$) is \textbf{absolutely continuous }%
if the orbital measure $\mu _{x_{1}}\ast \cdot \cdot \cdot \ast \mu _{x_{L}}$
on $G$ (resp., $\nu _{X_{1}}\ast \cdot \cdot \cdot \ast \nu _{X_{L}}$ on $%
\mathfrak{g}$) is absolutely continuous with respect to Haar measure on $G$
(resp., Lebesgue measure on $\mathfrak{g}$). This is the same as saying the
orbital measure has an $L^{1}$ density function.
\end{definition}

\subsection{Roots, Root vectors and Type}

We denote by $\Phi _{n}$ the set of roots of $\mathfrak{g}_{n}\mathfrak{:}$%
\begin{equation*}
\Phi _{n}=\{\pm e_{k},\pm e_{i}\pm e_{j}:1\leq i<j\leq n,1\leq k\leq n\}
\end{equation*}%
and let $\Phi _{n}^{+}$ denote the positive roots. Again, we may suppress
the subscript. For each root $\alpha \in \Phi _{n}$, we let $E_{\alpha }$
denote a root vector\textbf{\ }corresponding to $\alpha ,$ meaning that if $%
H\in \mathfrak{t}_{n}$, then $[H,E_{\alpha }]=i\alpha (H)E_{\alpha }$. (We
make the convention that roots are real valued.) We define a function $%
\alpha _{G}$ acting on $x\in T$ by the rule $\alpha _{G}(x):=\exp i\alpha
(X_{x})$.

We will take a collection of root vectors $\{E_{\alpha }:\alpha \in \Phi
^{+}\}$ that form a Weyl basis for $\mathfrak{g}$ (see \cite[p. 421]{He} or 
\cite[p. 290]{Va}) and set 
\begin{equation*}
RE_{\alpha }=\frac{E_{\alpha }+E_{-\alpha }}{2},IE_{\alpha }=\frac{E_{\alpha
}-E_{-\alpha }}{2i}.
\end{equation*}%
We also call these root vectors. These are vectors in $\mathfrak{g}$ such
that $E_{\alpha }=RE_{\alpha }+iIE_{\alpha }$. Being a Weyl basis, we have%
\begin{eqnarray}
\lbrack RE_{\alpha },RE_{\beta }] &=&cRE_{\alpha +\beta }+dRE_{\beta -a}
\label{Weyl} \\
-[IE_{\beta },RE_{\alpha }] &=&[RE_{\alpha },IE_{\beta }]=cIE_{\alpha +\beta
}+dIE_{\beta -a}  \notag \\
\lbrack IE_{\alpha },IE_{\beta }] &=&-cRE_{\alpha +\beta }+dRE_{\beta -a} 
\notag
\end{eqnarray}%
where $c,d\neq 0.$

By the annihhilating roots of $x\in T$ or $X\in \mathfrak{t}$ we mean the
sets 
\begin{eqnarray*}
\Phi _{x} &=&\{\alpha \in \Phi :\alpha _{G}(x)=1\}\text{ or} \\
\Phi _{X}^{\mathfrak{g}} &=&\{\alpha \in \Phi :\alpha (X)=0\}.
\end{eqnarray*}%
Equivalently,%
\begin{equation*}
\Phi _{x}=\{\alpha \in \Phi :\alpha (X_{x})=0\func{mod}2\pi \}.
\end{equation*}%
These are root subsystems.

We note that $\Phi _{x}=\Phi $ if and only if $x\in Z(G)$. If $\Phi _{x}$ is
empty, $x$ is called regular.

We point out that $\Phi _{X_{x}}^{\mathfrak{g}}\subseteq \Phi _{x}$ and this
inclusion can be proper. For example, if $x=(-1,...,-1)\in SO(2n+1)$, then
we may take $X_{x}=(\pi ,...,\pi ),$ so $\Phi _{x}=\{\pm e_{i}\pm
e_{j}:i\neq j\},$ while $\Phi _{X_{x}}^{\mathfrak{g}}=\{\pm
(e_{i}-e_{j}):i\neq j\}$. The centre of $SO(2n+1)$ is trivial.

We also let%
\begin{eqnarray*}
\mathcal{N}_{x} &=&\{RE_{\alpha },IE_{\alpha }:\alpha \in \Phi ^{+},\alpha
\notin \Phi _{x}\}\text{ and} \\
\mathcal{N}_{X}^{\mathfrak{g}} &=&\{RE_{\alpha },IE_{\alpha }:\alpha \in
\Phi ^{+},\alpha \notin \Phi _{X}^{\mathfrak{g}}\}\text{ ,}
\end{eqnarray*}%
the sets of non-annihilating root vectors. Of course, $\mathcal{N}%
_{x}\subseteq \mathcal{N}_{X_{x}}^{\mathfrak{g}}$. It is known that $\dim
C_{x}=\left\vert \mathcal{N}_{x}\right\vert $ and $\dim O_{X}=\left\vert 
\mathcal{N}_{X}^{\mathfrak{g}}\right\vert ,$ \cite[VI.4]{MT}.

\begin{example}
For the $x$ of (\ref{x}), $\Phi _{x}=$ $\Psi _{0}\tbigcup \Psi _{\pi
}\tbigcup \tbigcup_{\ell =1}^{m}\Psi _{\ell }$ where%
\begin{eqnarray*}
\Psi _{0} &=&\{\pm e_{k},e_{i}\pm e_{j}:1\leq i,j,k\leq u,i\neq j\} \\
\Psi _{\pi } &=&\{e_{i}\pm e_{j}:u<i\neq j\leq u+v\}\text{ and} \\
\Psi _{\ell } &=&\{e_{i}-e_{j}:u+v+s_{1}+\cdot \cdot \cdot +s_{\ell
-1}<i\neq j\leq u+v+s_{1}+\cdot \cdot \cdot +s_{\ell }\}
\end{eqnarray*}%
for $\ell =1,...,m$. The root vectors $RE_{e_{i}\pm e_{j}}$ and $RE_{e_{j}}$
with $i\leq u$ and $j>u$ are some examples of non-annihilating root vectors
of $x$.
\end{example}

There is an important geometric relationship between the absolute continuity
of a convolution product of orbital measures and the product of the
associated conjugacy classes or sum of orbits, as outlined in the next
result. See \cite[Lemma 7.6]{CJM}, \cite{Ra} and \cite[Prop. 2.2]{Wr} for
proofs.

\begin{proposition}
\label{charabscont}The orbital measure $\mu _{x_{1}}\ast \cdots \ast \mu
_{x_{L}}$ on $G$ (or $\nu _{X_{1}}\ast \cdot \cdot \cdot \ast \nu _{X_{L}}$
on $\mathfrak{g}$) is absolutely continuous with respect to Haar measure on $%
G$ (or Lebesgue measure on $\mathfrak{g}$) if and only if any of the
following hold:

\textrm{(i)} The set $\tprod_{i=1}^{L}C_{x_{i}}\subseteq G$ has non-empty
interior (resp., the set $\sum_{i=1}^{L}O_{X_{i}}\subseteq G$ has non-empty
interior);

\textrm{(ii)} The set $\tprod_{i=1}^{L}C_{x_{i}}\subseteq G$ has positive
measure (resp., the set $\sum_{i=1}^{L}O_{X_{i}}\subseteq G$ has positive
measure);

\textrm{(iii)} There exist $g_{i}\in G$ with $g_{1}=Id$, such that 
\begin{eqnarray*}
sp(Ad(g_{i})\mathcal{N}_{x_{i}} &:&i=1,\dots ,L)=\mathfrak{g}\text{ } \\
(\text{resp., }sp(Ad(g_{i})\mathcal{N}_{X_{i}}^{\mathfrak{g}} &:&i=1,\dots
,L)=\mathfrak{g}).
\end{eqnarray*}
\end{proposition}

We observe that this result implies that if $\mu _{x_{1}}\ast \cdots \ast
\mu _{x_{L}}$ (or $\mu _{X_{1}}\ast \cdot \cdot \cdot \ast \mu _{X_{L}})$ is
not absolutely continuous, then it is purely singular with respect to Haar
measure (or Lebesgue measure).

\section{Dominant Type and Eligibility}

\begin{definition}
Given $x\in T$ as in (\ref{x}), we will say that $x$ is of\textbf{\ type} 
\begin{equation*}
B_{u}\times D_{v}\times SU(s_{1})\times \cdot \cdot \cdot \times SU(s_{m})
\end{equation*}
as this is the Lie type of $\Phi _{x}$. Here by $B_{1}$ we mean the root
system $\{\pm e_{1}\},$ by $D_{2}$ we mean $\left\{ \pm e_{1}\pm
e_{2}\right\} ,$ and $\ B_{0},$ $D_{0},$ $D_{1},$ $SU(0)$ and $SU(1)$ are
empty (and may be omitted if this does not cause confusion).
\end{definition}

\begin{remark}
It seems to be more helpful to say Lie type $SU(s_{j})$, rather than type $%
A_{s_{j}-1},$ as the values $s_{j}$ arise naturally.
\end{remark}

There is a similar definition of type for elements of $\mathfrak{t}$, (see 
\cite{CJM}), but there is no $D_{v}$ component in that case.

\begin{example}
If $x$ is type $D_{1}$ or $SU(1),$ then $x$ is regular. The only element in $%
T_{n}$ of type $B_{n}$ is the Identity.
\end{example}

\begin{example}
When $x$ is type $B_{u}\times SU(s_{1})\times \cdot \cdot \cdot \times
SU(s_{m}),$ then $x$ and $X_{x}$ have the same type, but not otherwise. When 
$x$ and $X_{x}$ do not have the same type, we have $\mathcal{N}%
_{x}\subsetneq \mathcal{N}_{X_{x}}^{\mathfrak{g}},$ so $\dim C_{x}<\dim
O_{X_{x}}$. For example, when $x=(-1,...,-1)\in T_{n},$ then $x$ is type $%
D_{n},$ whereas $X_{x}$ is of type $SU(n)$. Thus $\dim C_{x}=2n,$ while $%
\dim O_{x}=n(n+1)$.
\end{example}

In \cite{CJM} the notion of dominant type for $X\in \mathfrak{t}$ was
defined: $X$ is of type $B_{u}\times SU(s_{1})\times \cdot \cdot \cdot
\times SU(s_{m})$ is of dominant type $B$ if $2u>\max s_{j}$ and otherwise $%
X $ is of dominant type $S$. Set $S_{X}^{\mathfrak{g}}=2u$ if $X$ is of
dominant type $B$ and $S_{X}^{\mathfrak{g}}=\max s_{j}$ otherwise.

The analogous notions for elements of $T$ are more complicated, as outlined
below.

\begin{definition}
\label{dom}Suppose $x\in T$ is of type $B_{u}\times D_{v}\times
SU(s_{1})\times \cdot \cdot \cdot \times SU(s_{m})$. Let $s=\max s_{j}$. We
will say 
\begin{equation*}
x\text{ is of \textbf{dominant type} }\left\{ 
\begin{array}{cc}
B & \text{if }u>v,2u+1>s \\ 
D & \text{if }v>u,2v>s \\ 
BD & \text{if }v=u,2v\geq s \\ 
S & \text{if }s\geq 2u+1,2v%
\end{array}%
\right. .
\end{equation*}%
Put 
\begin{equation*}
S_{x}=\text{ }\left\{ 
\begin{array}{cc}
2u+1 & \text{if }x\text{ is dominant type }B \\ 
2v & \text{if }x\text{ is dominant type }D \\ 
s & \text{if }x\text{ is dominant type }S \\ 
(2u+1,2v) & \text{if }x\text{ is dominant type }BD%
\end{array}%
\right. .
\end{equation*}%
In the case that $x$ is of dominant type $BD,$ we will let $S_{x}^{(1)}=2u+1$
and $S_{x}^{(2)}=2v$.
\end{definition}

It can be verified that each $x$ satisfies precisely one dominant type
criterion.

Note that if $x$ is of dominant type $B,D$ or $S,$ then $S_{x}$ is the
dimension of its largest eigenspace. When $x$ is type $BD,$ $S_{x}^{(1)}$ is
the dimension of the eigenspace corresponding to the eigenvalue $1$ and $%
S_{x}^{(2)}$ is the dimension of the eigenspace of the eigenvalue $-1,$ the
two largest eigenspaces. We will see that essentially we can treat $x$ as
either of dominant type $B$ or $D,$ depending on which is more helpful in
the circumstance.

In \cite{CJM}, the tuple $(X_{1},...,X_{L})\in \mathfrak{t}_{n}^{L},$ was
said to be eligible if 
\begin{equation}
\sum_{i=1}^{L}S_{X_{i}}^{\mathfrak{g}_{n}}\leq 2n(L-1).  \label{LieElig}
\end{equation}%
In defining eligibility in the group setting, it is helpful to first say
that a tuple $(x_{1},..,x_{L})$ has \textbf{parity }$p=1$ if there are an
odd number of $x_{j}$ that are of dominant type $D$ and \textbf{parity }$p=2$
if there are an even number.

\begin{definition}
\label{eligible}We will say that the tuple $(x_{1},...,x_{L})\in T_{n}^{L}$
of parity $p$ is \textbf{eligible} if one of the following (exclusive)
conditions hold.

(i) Any two or more of $x_{1},...,x_{L}$ are of dominant type $BD$ or $S.$

(ii) Only one $x_{j}$ is of dominant type $S$, none are of dominant type $BD 
$ and%
\begin{equation*}
\sum_{i=1}^{L}S_{x_{i}}\leq (2n+1)(L-1).
\end{equation*}

(iii) Only one $x_{j},$ say $x_{1},$ is of dominant type $BD$, none are of
dominant type $S$ and%
\begin{equation*}
S_{x_{1}}^{(p)}+\sum_{i=2}^{L}S_{x_{i}}\leq (2n+1)(L-1).
\end{equation*}

(iv) No $x_{j}$ is of dominant type $BD$ or $S$ and 
\begin{equation*}
\sum_{i=1}^{L}S_{x_{i}}\leq (2n+1)(L-1)+p-1.
\end{equation*}
\end{definition}

\begin{example}
A pair $(x,y)$ of dominant type $(B,B)$ (meaning both $x$ and $y$ are of
dominant type $B)$ is eligible if $2u_{1}+2u_{2}\leq 2n$, a pair of dominant
type $(D,D)$ is eligible if $2v_{1}+2v_{2}\leq 2n+2,$ and pairs of dominant
type $(B,D),(B,BD)$ or $(BD,D)$ are eligible if $2u_{1}+2v_{2}\leq 2n$. For
instance, if $x$ is type $B_{1}$ and $y$ is type $D_{2}$ in the Lie group $%
B_{2},$ then $(x,y)$ is not eligible. One can similarly determine the
requirements for types $(S,B)$ and $(S,D).$ All pair types $(S,S),$ $(S,BD)$
or $(BD,BD)$ are eligible.
\end{example}

\begin{example}
A central element in $G_{n}$ is type $B_{n}$. Thus if $x_{1}$ is central,
then $(x_{1},...,x_{L})$ is eligible if and only if $(x_{2},...,x_{L})$ is
eligible.
\end{example}

Our goal is to prove that eligibility characterizes absolute continuity. The
necessity of eligibility is easy to prove:

\begin{proposition}
\label{nec}Any absolutely continuous tuple in\ $T_{n}^{L}$ is eligible.
\end{proposition}

\begin{proof}
Assume $(x_{1},...,x_{L})$ is absolutely continuous, but not eligible. We
will use the fact that the absolute continuity of $\mu _{x_{1}}\ast \cdot
\cdot \cdot \ast \mu _{x_{L}}$ implies $\prod_{i=1}^{L}C_{x_{i}}$ has
non-empty interior, Proposition \ref{charabscont}.

First, consider the case that no $x_{j}$ is of dominant type $BD$ or $S$.
Let $y_{j}=g_{j}^{-1}x_{j}g_{j}$ be an arbitrary element of $C_{x_{j}}$ and
assume $V_{j}$ is the eigenspace of $y_{j}$ corresponding to the eigenvalue $%
\alpha _{j}=$ $1$ if $x_{j}$ is of dominant type $B$ and eigenvalue $-1$ if $%
x_{j}$ is of dominant type $D$. In either case, $V_{j}$ is the eigenspace of 
$y_{j}$ of dimension $S_{x_{j}}$.

Let $y=y_{1}\cdot \cdot \cdot y_{L}$ and note that if $v\in
\tbigcap_{j=1}^{L}V_{j}$, $v\neq 0,$ then $y(v)=\left( \prod_{i=1}^{L}\alpha
_{j}\right) v,$ so $v$ is an eigenvector of matrix $y$ with eigenvalue $%
\alpha =\prod_{i=1}^{L}\alpha _{j}$. Of course, $\alpha =1$ if there are an
even number of $x_{j}$ of dominant type $D$ and $\alpha =-1$ otherwise.

Now%
\begin{eqnarray*}
&&\dim \tbigcap\limits_{j=1}^{L}V_{j} \\
&=&\sum_{j=1}^{L}\dim V_{j}-\left( \dim (V_{1}+V_{2})+\dim ((V_{1}\tbigcap
V2\right) +V_{3})+\cdot \cdot \cdot +\dim \left(
\tbigcap\limits_{j=1}^{L-1}V_{j}+V_{L}\right) \\
&\geq &\sum_{j=1}^{L}S_{x_{j}}-(L-1)(2n+1).
\end{eqnarray*}

If there are an odd number of $x_{j}$ of dominant type $D,$ then the failure
of eligibility implies $\dim \tbigcap_{j=1}^{L}V_{j}\geq 1$. This shows that
any $y\in \prod_{i=1}^{L}C_{x_{i}}$ has a non-zero eigenvector with
eigenvalue $\alpha =-1$ and that is impossible for a set with non-empty
interior.

If there are an even number of $x_{j}$ of dominant type $D,$ then the
failure of eligibility implies $\dim \tbigcap_{j=1}^{L}V_{j}\geq 2$, so any
such $y$ has two linearly independent eigenvectors with eigenvalue $1$.
Every element of $SO(2n+1)$ has one eigenvector with eigenvalue $1$, but a
set with non-empty interior cannot have a second linearly independent
eigevector with eigenvalue $1,$ so again we obtain a contradiction.

Next, suppose there is one $x_{j},$ say $x_{1},$ of dominant type $BD$ and
none of dominant type $S$. Assume there are an odd number of $x_{j}$ of
dominant type $D$. As before, let $y_{j}$ be an arbitrary element of $%
C_{x_{j}}$ and $y=y_{1}\cdot \cdot \cdot y_{L}$. Let $V_{1}$ be the
eigenspace of $y_{1}$ corresponding to the eigenvalue $1$ having dimension $%
S_{x_{1}}^{(1)},$ and let $V_{j}$ be eigenspace of $y_{j}$ of dimension $%
S_{x_{j}}$ for $j\neq 1$. The calculations as above and failure of
eligibility implies 
\begin{equation*}
\dim \tbigcap\limits_{j=1}^{L}V_{j}\geq
S_{x_{1}}^{(1)}+\sum_{j=2}^{L}S_{x_{j}}-(L-1)(2n+1)\geq 1\text{.}
\end{equation*}%
Moreover, any non-zero $v\in \tbigcap_{j=1}^{L}V_{j}$ is an eigenvector of $%
y $ corresponding to the eigenvalue $(-1)^{\#x_{j}\text{ dominant }D}=-1$ as
we have assumed there are an odd number of $x_{j}$ of dominant type $D$.
Thus every element of $\prod_{i=1}^{L}C_{x_{i}}$ has an eigenvector with
eigenvalue $\alpha =-1$ and that is impossible for a set with non-empty
interior.

If there are an even number of $x_{j}$ of dominant type $D$ the arguments
are similar, but we begin with $V_{1}$ the eigenspace of $Y_{1}$ with
eigenvalue $-1$. This space has dimension $S_{x_{1}}^{(2)}$. In this case,
any non-zero $v\in \tbigcap_{j=1}^{L}V_{j}$ is an eigenvector of $Y$
corresponding to the eigenvalue $(-1)(-1)^{\#x_{j}\text{ dominant }D}=-1$
and again the failure of eligibility implies that $\tbigcap_{j=1}^{L}V_{j}$
has dimension at least one which is a contradiction.

Finally, if $x_{1}$ is of dominant type $S,$ with eigenvalue $\alpha
_{1}\neq \pm 1$ having eigenspace of dimension $S_{x_{1}}$, and no other $%
x_{j}$ is type $\,S$ or $BD$, then we begin with $V_{1}$ the eigenspace of $%
y_{1}$ corresponding to $\alpha _{1}$. Arguing as above, we deduce that
non-eligbility implies every element of $\prod_{i=1}^{L}C_{x_{i}}$ has an
eigenvector with eigenvalue equal to either $\pm \alpha _{1},$ again
impossible for a set with non-empty interior.

As we assumed $(x_{1},...,x_{L})$ was not eligible, these are the only cases
to consider.
\end{proof}

\section{Preliminary results towards proving absolute continuity}

Our proof that eligible tuples are absolutely continuous will proceed by
induction on the rank of the Lie group. This will require associating each $%
x $ in the torus of \ $G_{n}$ with some $x^{\prime }$ in the torus of $%
G_{n-1}$, as was done for elements of the torus of the Lie algebra in \cite%
{CJM}.

\begin{notation}
Assume $x\in T_{n}$ is as in (\ref{x}). We call $x^{\prime }$ $\in T_{n-1}$
the \textbf{reduction of }$x$ when
\end{notation}

\begin{equation*}
x^{\prime }=\left\{ 
\begin{array}{cc}
(\underbrace{1,...,1}_{u-1},\underbrace{-1,...,-1}_{v},\underbrace{%
e^{ia_{1}},...,e^{ia_{1}}}_{s_{1}},...,\underbrace{e^{ia_{m}},...,e^{ia_{m}}}%
_{s_{m}}) & \text{ if }x\text{ is dominant }B\text{ or }BD \\ 
(\underbrace{1,...,1}_{u},\underbrace{-1,...,-1}_{v-1},\underbrace{%
e^{ia_{1}},...,e^{ia_{1}}}_{s_{1}},...,\underbrace{e^{ia_{m}},...,e^{ia_{m}}}%
_{s_{m}}) & \text{ if }x\text{ is dominant }D \\ 
(\underbrace{1,...,1}_{u},\underbrace{-1,...,-1}_{v},\underbrace{%
e^{ia_{1}},...,e^{ia_{1}}}_{s_{1}-1},...,\underbrace{%
e^{ia_{m}},...,e^{ia_{m}}}_{s_{m}}) & 
\begin{array}{c}
\text{ if }x\text{ is dominant }S\text{ and } \\ 
s_{1}=\max s_{j}%
\end{array}%
\end{array}%
\right. .
\end{equation*}

We can embed $\mathfrak{t}_{n-1}$ into $\mathfrak{t}_{n}$ by taking the
standard basis vectors $e_{1},...,e_{n}$ for $\mathfrak{t}_{n}$ and omiting
the first for $\mathfrak{t}_{n-1}$. This gives a natural embedding of $\Phi
_{n-1}$ into $\Phi _{n},$ an embedding of $\mathfrak{g}_{n-1}$ into $%
\mathfrak{g}_{n},$ $G_{n-1}$ into $G_{n}$ and $T_{n-1}$ into $T_{n}$. It is
in this way that we view $x^{\prime }$ as an element of $T_{n-1}$.

Observe that if $x$ is of dominant type $BD,$ then $x^{\prime }$ is either
of dominant type $D$ or $S$. Also, $x$ is central if and only if $x^{\prime
} $ is central. We record here some other simple facts.

\begin{lemma}
\label{1}(a) If $x$ is either of dominant type $BD$ or\ $S,$ then $\mu
_{x}^{2}\in L^{2}$.

(b) If $x$ is of dominant type $BD$ or $S,$ then $\mu _{x^{\prime }}^{2}\in
L^{2}$.
\end{lemma}

\begin{proof}
(a) This follows directly from Theorem 9.1(B) of \cite{Adv} with the
observation that if $x$ is of dominant type $BD,$ then $x$ is either type $%
B_{n/2}\times D_{n/2}$ or type $B_{u}\times D_{u}\times SU(s_{1})\times
\cdot \cdot \cdot \times SU(s_{m})$ where $u<n/2$.

(b) If $x$ as in (\ref{x}) is of dominant type $BD$, then $x^{\prime }$ is
type $B_{u-1}\times D_{u}\times SU(s_{1})\times \cdot \cdot \cdot \times
SU(s_{m})$. Unless $\sum s_{j}=0,1,$ we are in the `else' case in the
notation of \cite[Thm. 9.1(B)]{Adv} and consequently $\mu _{x^{\prime
}}^{2}\in L^{2}$. If $\sum s_{j}=0$ (or $1),$ then $u=n/2$ (resp., $%
u=(n-1)/2)$ and one can check from \cite[Thm. 9.1(B)]{Adv} that both
situations imply $\mu _{x^{\prime }}^{2}\in L^{2}$.

The reasoning is similar if $x$ is of dominant type $S$ and is left to the
reader.
\end{proof}

\begin{corollary}
(a) If $x$ and $y$ are both of dominant types $BD$ or $S,$ then $\mu
_{x}\ast \mu _{y}$ and $\mu _{x^{\prime }}\ast \mu _{y^{\prime }}\in L^{2}.$

(b) If $x$ is of dominant type $B$ or $D$ and changes type upon reduction,
then $\mu _{x^{\prime }}^{2}\in L^{2}$.
\end{corollary}

\begin{proof}
(a) This is immediate from parts (a) and (b) of the previous Lemma upon
applying Plancherel's theorem and Holder's inequality.

(b) For this we merely need to observe that the change of dominant type of \ 
$x$ upon reduction implies $x^{\prime }$ is either of dominant type $BD$ or $%
S$.
\end{proof}

An important property for our induction argument is that eligibility is
preserved under reduction.

\begin{proposition}
\label{reductioneligible}If $(x_{1},...,x_{L})$ in $G_{n}$ is eligible, then
so is the reduced tuple $(x_{1}^{\prime },...,x_{L}^{\prime })$ in $G_{n-1}$.
\end{proposition}

\begin{proof}
Note that Lemma \ref{1} and its Corollary imply that if $x$ is of dominant
type $BD$ or $S,$ or $x$ and $x^{\prime }$ are different dominant types,
then $\mu _{x^{\prime }}^{2}\in L^{2}$. Consequently, $\mu _{x_{1}^{\prime
}}\ast \cdot \cdot \cdot \ast \mu _{x_{L}^{\prime }}\in L^{2}$ $\subseteq
L^{1},$ and hence $(x_{1}^{\prime },...,x_{L}^{\prime })$ is eligible, if
any of the following situations occur:

(i) two or more $x_{j}$ are of dominant type $BD$ or $S$;

(ii) two or more $x_{j}$ switch dominant type upon reduction;

(iii) one $x_{j}$ is of dominant type $BD$ or $S$ and another switches
dominant type upon reduction.

Thus there are two cases that we need to analyze.

Case 1: All $x_{j}$ are of dominant type $B$ or $D$ and at most one switches
dominant type upon reduction.

Case 2: One $x_{j}$ is of dominant type $BD$ or $S$ and no other $x_{i}$
switches dominant type upon reduction.

We remark that it is a routine exercise to check that if $x$ is of dominant
type $B$ or $D$ and $x$ does not switch dominant type upon reduction, then $%
S_{x^{\prime }}=S_{x}-2$, while if $x$ switches dominant type, then $%
S_{x^{\prime }}\leq S_{x}-1$ (where if $x^{\prime }$ is of dominant type $BD$
by `$S_{x^{\prime }}\leq S_{x}-1$' we mean both coordinates of $S_{x^{\prime
}}$ are dominated by $S_{x}-1$). If $x$ is of dominant type $S,$ then $%
S_{x^{\prime }}\leq S_{x}$ (again, meaning both coordinates if $x^{\prime }$
is of dominant type $BD)$ and if $x$ is of dominant type $BD,$ then $%
S_{x^{\prime }}=S_{x}^{(2)}=\min (S_{x}^{(1)},S_{x}^{(2)}).$ (In this final
situation, $x^{\prime }$ cannot be type $BD$.)

\medskip

Proof of Case 1: Regardless of the parity of the tuple, eligibility
certainly implies 
\begin{equation*}
\sum_{i=1}^{L}S_{x_{i}}^{{}}\leq (2n+1)(L-1)+1\text{. }
\end{equation*}%
Assume $x_{1}$ is the one that switches dominant type (if any do). Then 
\begin{eqnarray*}
\sum_{i=1}^{L}S_{x_{i}^{\prime }} &\leq
&S_{x_{1}}^{{}}-1+\sum_{i=2}^{L}(S_{x_{i}}^{{}}-2)=%
\sum_{i=1}^{L}S_{x_{i}}^{{}}-2(L-1)-1 \\
&\leq &(2(n-1)+1)(L-1),
\end{eqnarray*}%
so $(x_{1}^{\prime },...,x_{L}^{\prime })$ is eligible in $G_{n-1}$.

\medskip

Proof of Case 2: If $x_{1}$ is of dominant type $BD,$ then eligibility
implies%
\begin{equation*}
S_{x_{1}}^{(p)}+\sum_{i=2}^{L}S_{x_{i}}^{{}}\leq (2n+1)(L-1)\text{ }
\end{equation*}%
where $p$ is the parity of the tuple. As $S_{x_{i}^{\prime }}=S_{x_{i}}-2$
for $i\neq 1$ and $S_{x_{1}^{\prime }}\leq S_{x_{1}}^{(p)},$ we have%
\begin{equation*}
\sum_{i=1}^{L}S_{x_{i}^{\prime }}\leq \sum_{i=1}^{L}S_{x_{i}}-2(L-1)\leq
(2(n-1)+1)(L-1),
\end{equation*}%
and hence the tuple is eligible.

The arguments are similar if $x_{1}$ is of dominant type $S$.
\end{proof}

Our proof of the sufficiency of eligibility for absolute continuity will
make heavy use of the following Proposition, which we refer to as the
general strategy. Its proof relies on Proposition \ref{charabscont} and is
essentially the same as given in Proposition 5.6 in \cite{CJM} in the Lie
algebra setting and is omitted.

\begin{notation}
For $x\in G$ or $X\in \mathfrak{g}$ put%
\begin{equation*}
\Omega _{x}=\mathcal{N}_{x}\diagdown \mathcal{N}_{x^{\prime }}\text{ and }%
\Omega _{X}^{\mathfrak{g}}=\mathcal{N}_{X}^{\mathfrak{g}}\diagdown \mathcal{N%
}_{X^{\prime }}^{\mathfrak{g}}
\end{equation*}%
where we use the natural embedding described above.
\end{notation}

\begin{proposition}
\label{generalstrategy}\textrm{\ }(General strategy) Let $L\geq 2$. Let $%
x_{i}\in T_{n}$ for $i=1,\dots ,L$ and assume $(x_{1}^{\prime },\dots
,x_{L}^{\prime })$ is an absolutely continuous tuple in $G_{n-1}$. Suppose 
\begin{equation}
\Omega \subseteq \{RE_{\alpha },IE_{\alpha }:\alpha \in \Phi
_{n}^{+}\diagdown \Phi _{n-1}^{+}\},  \label{Omega}
\end{equation}%
$\Omega $ contains each $\Omega _{x_{i}}$ and $\Omega $ has the property
that $ad(H)(\Omega )\subseteq sp\Omega $ whenever $H\in \mathfrak{g}_{n-1}$.
Fix $\Omega _{0}\subseteq \Omega _{x_{L}}$ and assume there exist $%
g_{1},\dots ,g_{L-1}\in G_{n-1}$ and $M\in \mathfrak{g}_{n}$ such that

\textrm{(i)} $sp(Ad(g_{i})(\Omega _{x_{i}}),\Omega _{x_{L}}\backslash \Omega
_{0}:i=1,\dots ,L-1)=sp\Omega ;$

\textrm{(ii)} $ad^{k}(M):\mathcal{N}_{x_{L}}\diagdown \Omega _{0}\rightarrow
sp(\Omega ,\mathfrak{g}_{n-1})$ for all positive integers $k$; and

\textrm{(iii)} The span of the projection of $Ad(\exp tM)(\Omega _{0})$ onto
the orthogonal complement of $sp(\mathfrak{g}_{n-1},\Omega )$ in $\mathfrak{g%
}_{n}$ is a surjection for all small $t\neq 0.$

Then $(x_{1},\dots ,x_{L})$ is an absolutely continuous tuple in $G_{n}$.
\end{proposition}

The following elementary linear algebra result can be helpful in verifying
the hypotheses of this proposition. Its proof is essentially the same as
that of Lemma 5.7 of \cite{CJM}.

\begin{lemma}
\label{elem}Let $\Xi _{i},$ $i=1,...,L,$ $L\geq 2,$ be given and suppose $%
\Omega $ is as (\ref{Omega}). Assume $\Omega \supseteq \tbigcup_{i=1}^{L}\Xi
_{i}$ and has the property that $ad(H)(\Omega )\subseteq sp\Omega $ whenever 
$H\in \mathfrak{g}_{n-1}$. Fix $\Omega _{0}\subseteq \Xi _{L}$ and $\Omega
_{\ast }\subseteq \left( \Xi _{L}\tbigcap \tbigcup_{i=1}^{L-1}\Xi
_{i}\right) \diagdown \Omega _{0}$, and assume that for some $H_{0}\in 
\mathfrak{g}_{n-1},$%
\begin{equation*}
sp\left( ad(H_{0})\Omega _{\ast },\Xi _{L}\diagdown \Omega
_{0},\tbigcup\limits_{i=1}^{L-1}\Xi _{i}\diagdown \Omega _{\ast }\right)
=sp\Omega .
\end{equation*}%
Then, for small $t\neq 0,$%
\begin{equation*}
sp\left( Ad(\exp tH_{0})\tbigcup\limits_{i=1}^{L-1}\Xi _{i},\Xi
_{L}\diagdown \Omega _{0}\right) =sp\Omega .
\end{equation*}
\end{lemma}

\section{Characterizing absolute continuity}

\subsection{Proof of the characterization theorem}

\begin{theorem}
\label{main}Assume $x_{j},$ $j=1,..,L$ are non-central, torus elements of $%
G_{n}$. The tuple $(x_{1},...,x_{L})$ is eligible if and only if $\mu
_{x_{1}}\ast \cdot \cdot \cdot \ast \mu _{x_{L}}$ is absolutely continuous.
\end{theorem}

We already saw the necessity of eligibility in Proposition \ref{nec}, thus
we only need to prove sufficiency. This will be a proof by induction on $n,$
the rank of $G_{n},$ and we begin with the base case, $n=2$.

\begin{lemma}
\label{ind}All eligible tuples of non-central torus elements in $G_{2}$ are
absolutely continuous.

\begin{proof}
It is shown in Theorem 9.1 of \cite{Adv} that if $x\in G_{2}\diagdown
Z(G_{2})$, then $\mu _{x}^{2}\in L^{2},$ except if $x$ is type $D_{2}$ when $%
\mu _{x}^{4}\in L^{2}$. Thus, if $L\geq 4,$ it follows from Plancherel's
theorem and Holder's inequality that $(x_{1},...,x_{L})$ is absolutely
continuous. Similarly, if $L=3$ and only two $x_{j}$ are type $D_{2},$ then
the triple is absolutely continuous. If all three $x_{j}$ are type $D_{2},$
the triple is not eligible. That proves all eligible triples are absolutely
continuous.

Likewise, any pair with neither $x_{1}$ nor $x_{2}$ of type $D_{2}$ is
absolutely continuous. The only eligible pairs for which $x_{1}$ is type $%
D_{2}$ are the pairs $(x_{1},x_{2})$ where $x_{2}$ is regular, and Theorem
1.3 of \cite{Wr} implies that such pairs are absolutely continuous.
\end{proof}
\end{lemma}

The major work in proving the theorem is done in the next lemma.

\begin{lemma}
\label{MainLemma}Suppose $n\geq 3$. Assume that all eligible tuples of
non-central elements in $G_{n-1}$ are absolutely continuous and that $%
(x_{1},...,x_{L})$ is eligible in $G_{n}$.

(a) If all $x_{j}$ are either of dominant type $B$ or $D,$ and the number of 
$x_{j}$ of dominant type $D$ is even, then $(x_{1},...,x_{L})$ is absolutely
continuous.

(b) If one $x_{j}$ is of dominant type $S$ and none are of dominant type $%
BD, $ then $(x_{1},...,x_{L})$ is absolutely continuous.
\end{lemma}

\begin{proof}
For both (a) and (b) we will use the general strategy, Proposition \ref%
{generalstrategy}, taking 
\begin{equation*}
\Omega :=\{RE_{\alpha },IE_{\alpha }:\alpha \in \Phi _{n}^{+}\diagdown \Phi
_{n-1}^{+}\}
\end{equation*}%
\begin{equation*}
=\{FE_{e_{1}},FE_{e_{1}\pm e_{j}}:j=2,...,n,F=R,I\}.
\end{equation*}%
Note that the orthogonal complement of $sp(\mathfrak{g}_{n-1},\Omega )$ in $%
\mathfrak{g}_{n}$ is the one dimensional space spanned, for example, by the
torus element $[RE_{e_{1}},IE_{e_{1}}]$ or $%
[RE_{e_{1}+e_{n}},IE_{e_{1+}e_{n}}]$.

Since $(x_{1},...,x_{L})$ is eligible in $G_{n}$, Proposition \ref%
{reductioneligible} shows that the reduced tuple, $(x_{1}^{\prime
},...,x_{L}^{\prime })$ in $G_{n-1},$ is eligible. By the hypothesis of the
lemma, it is an absolutely continuous tuple.

\smallskip

(a) Without loss of generality, we can assume $x_{j}$ are of dominant type $%
D $ for $j=1,...,m$ and $x_{j}$ are of dominant type $B$ for $j=m+1,...,L,$
with $m$ even. In this situation, eligibility simplifies to 
\begin{equation}
\sum_{i=1}^{m}2v_{i}+\sum_{i=m+1}^{L}2u_{i}\leq 2n(L-1)+m.  \label{El1}
\end{equation}

\medskip We begin with the case $m\neq 0,$ so $m\geq 2$. We will put $\Omega
_{0}=\{RE_{e_{1}},IE_{e_{1}}\}$ and we have 
\begin{equation*}
\ \Omega _{x_{i}}^{{}}=\left\{ 
\begin{array}{cc}
\{FE_{e_{1}},FE_{e_{1}\pm e_{j}}:j\in J_{i,}F=R,I\} & \text{for }i=1,...,m
\\ 
\{FE_{e_{1}\pm e_{j}}:j\in J_{i},F=R,I\} & \text{for }i=m+1,...,L%
\end{array}%
\right. ,
\end{equation*}%
where the sets $J_{i}\subseteq \{2,...,n\},$ $\left\vert J_{i}\right\vert
=n-v_{i}$ for $i=1,...,m$, and $\left\vert J_{i}\right\vert =n-u_{i}$ for $%
i=m+1,...,L$. For notational ease, we will write $\Omega _{i}=\Omega
_{x_{i}}^{{}}$.

Clearly, $\Omega _{i}\subseteq \Omega ,$ $\Omega _{0}\subseteq $ $\Omega
_{1}\cap \Omega _{2}$ and $ad(H)(\Omega )\subseteq sp\Omega $ for all $H\in 
\mathfrak{g}_{n-1}$.

By replacing $x_{i}$ by Weyl conjugates that permute the appropriate
indices, if necessary, (or, equivalently, replacing $\Omega _{i}$ by $%
Ad(g_{i})\Omega _{i}$ where $g_{i}\in G_{n-1}$ is a suitable Weyl
conjugation), there is no loss of generality in assuming the sets $J_{i}$
are disjoint (as much as possible).

In particular, if $\sum_{i=1}^{L}\left\vert J_{i}\right\vert \geq n-1,$ then 
$\tbigcup_{i=1}^{L}J_{i}=\{2,...,n\}$ and in that case since $RE_{e_{1}}$, $%
IE_{e_{1}}\in \Omega _{2},$ 
\begin{equation*}
(\Omega _{1}\backslash \Omega _{0})\tbigcup \tbigcup\limits_{i=2}^{L}\Omega
_{i}=\Omega ,
\end{equation*}%
so property (i) of the general strategy, Proposition \ref{generalstrategy},
holds.

Put $M=RE_{e_{1}}.$ Property (iii) of the general strategy holds since 
\begin{equation*}
Ad(\exp tM)IE_{e_{1}}=a_{t}IE_{e_{1}}+tb_{t}[RE_{e_{1}},IE_{e_{1}}]
\end{equation*}%
where $b_{t}\rightarrow b\neq 0$ as $t\rightarrow 0$. Moreover, $%
ad(M)(FE_{e_{1}\pm e_{j}})\in spFE_{e_{j}}$ for $j\neq 1$, $%
ad(M)(FE_{e_{i}\pm e_{j}})=0$ if neither $i,j$ $=1$ and $ad(M)(FE_{e_{j}})%
\in sp(FE_{e_{1}\pm e_{j}})$ if $j\neq 1$. It follows that 
\begin{equation*}
ad^{k}(M):\mathcal{N}_{x_{1}}\diagdown \Omega _{0}\subseteq sp(\Omega ,%
\mathfrak{g}_{n-1}),
\end{equation*}%
so (ii) of the general strategy holds. Consequently, Proposition \ref%
{generalstrategy} implies $(x_{1},...,x_{L})$ is an absolutely continuous
tuple.

\smallskip

So assume $\sum \left\vert J_{i}\right\vert <n-1$. The number of indices
from $\{2,...,n\}$ that are missing from $\tbigcup_{i=1}^{L}J_{i}$ is equal
to 
\begin{equation*}
n-1-\left( \sum_{i=1}^{m}(n-v_{i})+\sum_{i=m+1}^{L}(n-u_{i})\right)
=n(1-L)-1+\sum_{i=1}^{m}v_{i}+\sum_{i=m+1}^{L}u_{i}
\end{equation*}%
and the eligibility assumption (\ref{El1}) ensures this is bounded above by 
\begin{equation*}
n(1-L)-1+n(L-1)+\frac{m}{2}=\frac{m}{2}-1\text{.}
\end{equation*}%
Thus $(\Omega _{1}\backslash \Omega _{0})\tbigcup \tbigcup_{i=2}^{L}\Omega
_{i}$ contains all of $\Omega ,$ except possibly as many as $m/2-1$ of the
pairs $FE_{e_{1}\pm e_{j}}.$ Moreover, we must have $m\geq 4$ since we are
assuming that some indices are missing. There is no loss of generality in
assuming that it is the indices $j\in \{2,...,K\}$ with $K\leq m/2$ which
have the property that $FE_{e_{1}\pm e_{j}}$ does not belong to $(\Omega
_{1}\backslash \Omega _{0})\tbigcup \tbigcup_{i=2}^{L}\Omega _{i}$.

The idea is now to use the fact that $FE_{e_{1}}$ occurs in each of $\Omega
_{1},...,\Omega _{m}$. One of these copies will be needed to `cover' $\Omega
,$ one will be used to obtain the torus element as above, and we will see
that the other $m-2$ copies can be used to obtain the pairs $FE_{e_{1}\pm
e_{j}}\in \Omega $ that are not present in $(\Omega _{1}\backslash \Omega
_{0})\tbigcup \tbigcup_{i=2}^{L}\Omega _{i}$.

An important observation towards this is that since $%
ad(RE_{e_{j}})(FE_{e_{1}})=cFE_{e_{1}+e_{j}}+dFE_{e_{1}-e_{j}}$ and $%
ad(IE_{e_{j}})(FE_{e_{1}})=-cF^{\prime }E_{e_{1}+e_{j}}+dF^{\prime
}E_{e_{1}-e_{j}}$ where $F^{\prime }=I$ if $F=R$ and vice versa, we have%
\begin{eqnarray*}
Ad(\exp tRE_{e_{j}})(FE_{e_{1}})
&=&a_{t}FE_{e_{1}}+t(c_{t}FE_{e_{1}+e_{j}}+d_{t}FE_{e_{1}-e_{j}})\text{ and}
\\
Ad(\exp tIE_{e_{j}})(FE_{e_{1}}) &=&a_{t}^{\prime }F^{\prime
}E_{e_{1}}+t(-c_{t}F^{\prime }E_{e_{1}+e_{j}}+d_{t}F^{\prime
}E_{e_{1}-e_{j}})\text{ }
\end{eqnarray*}%
where $c_{t},d_{t}$ converge to $c,d\neq 0$ respectively as $t\rightarrow 0$%
. Since the pairs 
\begin{equation*}
c_{t}FE_{e_{1}+e_{j}}+d_{t}FE_{e_{1}-e_{j}},-c_{t}FE_{e_{1}+e_{j}}+d_{t}FE_{e_{1}-e_{j}}
\end{equation*}%
are linearly independent for small $t\neq 0$, it follows that 
\begin{equation*}
sp(FE_{e_{1}},Ad(\exp tRE_{e_{j}})(FE_{e_{1}}),Ad(\exp
tIE_{e_{j}})(FE_{e_{1}}):F=R,I)
\end{equation*}%
\begin{equation*}
=sp(FE_{e_{1}},FE_{e_{1}\pm e_{j}}:F=R,I).
\end{equation*}%
Moreover, $Ad(\exp tF^{\prime }E_{e_{j}})(FE_{e_{1}\pm e_{k}})=FE_{e_{1}\pm
e_{k}}$ if $F^{\prime }=R,I$ and $j\neq k$.

It follows from these observations that for small $t\neq 0$ and $j\in
\{2,...,K\},$%
\begin{equation*}
sp(FE_{e_{1}},Ad(\exp tRE_{e_{j}})(\Omega _{2j-1}),Ad(\exp
tIE_{e_{j}})(\Omega _{2j}):F=R,I)
\end{equation*}%
\begin{equation*}
=sp(FE_{e_{1}},FE_{e_{1}\pm e_{j}},FE_{e_{1}\pm e_{i}}:i\in J_{2j-1}\cup
J_{2j},F=R,I).
\end{equation*}%
Since $FE_{e_{1}}\in \Omega _{2},$ we conclude that $sp(\Omega )$ is equal
to 
\begin{equation*}
sp(\Omega _{1}\backslash \Omega _{0},\Omega _{2},\Omega _{2K+1},...,\Omega
_{L},Ad(\exp tRE_{e_{j}})(\Omega _{2j-1}),Ad(\exp tIE_{e_{j}})(\Omega
_{2j}):j=2,...,K).
\end{equation*}%
That establishes property (i) of the general strategy. We again take $%
M=RE_{e_{1}}$ to complete the argument as we did in the case when there were
no missing indices.

\smallskip \smallskip

Now assume $m=0$. Here we will appeal to the Lie algebra argument of \cite%
{CJM}. Consider $X_{i}=X_{x_{i}}$ for $i=1,...,L$. These are of dominant
type $B$ in the Lie algebra $\mathfrak{g}_{n}$. They are eligible in the Lie
algebra sense (\ref{LieElig}) since 
\begin{equation*}
\sum_{i=1}^{L}S_{X_{i}}^{\mathfrak{g}_{n}}=\sum_{i=1}^{L}(S_{X_{i}}-1)\geq
(2n+1)(L-1)+1-L=2n(L-1).
\end{equation*}

The proof given in Proposition 5.8 of \cite{CJM} for $\mathfrak{g}_{n}$ of
type $B_{n}$ and $L=2,$ or in Section 6 of \cite{CJM} for $L\geq 3,$ shows
that under the eligibility assumption there exist $g_{i}\in G_{n-1}$ such
that 
\begin{equation*}
sp(Ad(g_{i})(\Omega _{x_{i}}^{\mathfrak{g}}),\Omega _{x_{L}}^{\mathfrak{g}%
}\backslash \Omega _{0}:i=1,\dots ,L-1)=sp\Omega ^{\mathfrak{g}}
\end{equation*}%
where 
\begin{eqnarray*}
\Omega ^{\mathfrak{g}} &:&=\{RE_{\alpha },IE_{\alpha }:\alpha \in \Phi
_{n}^{+\mathfrak{g}}\diagdown \Phi _{n-1}^{+\mathfrak{g}}\}=\Omega \text{
and } \\
\Omega _{0} &=&\{RE_{e_{1}+e_{n}},IE_{e_{1}+e_{n}}\}\subseteq \Omega
_{x_{L}}^{\mathfrak{g}}.
\end{eqnarray*}%
Since $\Omega _{x_{i}}^{\mathfrak{g}}=\Omega _{x_{i}}$, this establishes
property (i) of the general strategy and taking $M=RE_{e_{1}+e_{n}}$ we
complete the proof in a similar fashion to before.

That completes the proof of (a).

\medskip

(b) In this case, we can assume $x_{j}$ are of dominant type $D$ for $%
j=1,...,m$, $x_{j}$ are of dominant type $B$ for $j=m+1,...,L-1,$ and $x_{L}$
is of dominant type $S$. Eligibility tells us that 
\begin{equation}
\sum_{i=1}^{m}2v_{i}+\sum_{i=m+1}^{L-1}2u_{i}+s_{L}\leq 2n(L-1)+m.
\label{El2}
\end{equation}

First, assume $m\geq 1$. As before, we have%
\begin{equation*}
\ \Omega _{x_{i}}=\left\{ 
\begin{array}{cc}
\{FE_{e_{1}},FE_{e_{1}\pm e_{j}}:j\in J_{i},F=R,I\} & \text{for }i=1,...,m
\\ 
\{FE_{e_{1}\pm e_{j}}:j\in J_{i},F=R,I\} & \text{for }i=m+1,...,L-1%
\end{array}%
\right. ,
\end{equation*}%
where the sets $J_{i}\subseteq \{2,...,n\},$ $\left\vert J_{i}\right\vert
=n-v_{i}$ for $i=1,...,m$, and $\left\vert J_{i}\right\vert =n-u_{i}$ for $%
i=m+1,...,L-1$. But, in this case%
\begin{equation*}
\Omega _{x_{L}}=\{FE_{e_{1}},FE_{e_{1}-e_{j}},FE_{e_{1}+e_{k}},F=R,I:j\in
J_{L},k\geq 2\}
\end{equation*}%
where $|J_{L}|=n-S_{L}$. There is no loss of generality in assuming the sets 
$J_{i}$ are disjoint (as much as possible), so the cardinality of $%
\tbigcup_{i=1}^{L}J_{i}$ is the minimum of $n-1$ and%
\begin{equation*}
\sum_{i=1}^{m}(n-v_{i})+\sum_{i=m+1}^{L-1}(n-u_{i})+n-s_{L}.
\end{equation*}%
If $\left\vert \tbigcup_{i=1}^{L}J_{i}\right\vert =n-1,$ then already 
\begin{equation*}
\Omega =\tbigcup\limits_{i=1}^{L-1}\Omega _{x_{i}}\tbigcup \left( \Omega
_{x_{L}}\diagdown \{FE_{e_{1}}\}\right)
\end{equation*}%
and we can directly apply the general strategy with $M=RE_{e_{1}}$, as in
the first case.

\smallskip

So assume otherwise and choose $m-1$ indices from $\{2,...,n\}\diagdown
\tbigcup_{i=1}^{L}J_{i}$ (or all these indices if there are less than $m-1$%
), say the indices $k_{i}$ for $i=2,...,m^{\prime }$.

For $j\neq 1,$ 
\begin{equation*}
Ad(\exp
tRE_{e_{j}})(FE_{e_{1}})=a_{t}FE_{e_{1}}+tb_{t}FE_{e_{1}+e_{j}}+tc_{t}FE_{e_{1}-e_{j}}
\end{equation*}%
where $a_{t},b_{t},c_{t}$ converge to non-zero constants as $t\rightarrow 0$%
. For $\ell \neq j,$%
\begin{equation*}
Ad(\exp tRE_{e_{j}})(FE_{e_{1}\pm e_{\ell }})=FE_{e_{1}\pm e_{\ell }}.
\end{equation*}%
Put $g_{i}=\exp tRE_{e_{k_{i}}}\in G_{n-1}$ for $i=2,...,m^{\prime }$. Let $%
J_{0}\subseteq \{2,...,n\}$ consist of the union of the sets $J_{j},$ $%
j=1,...,L$, together with the additional indices $k_{i},$ $i=2,...,m^{\prime
}$. Since $\Omega _{x_{L}}$ contains all the root vectors of the form $%
FE_{e_{1}+e_{k}}$ and $\Omega _{x_{1}}$ contains $FE_{e_{1}},$ it follows
that for small $t\neq 0$,%
\begin{eqnarray*}
&&sp\left( \Omega _{x_{1}},\tbigcup\limits_{j=M+1}^{L-1}\Omega
_{x_{j}},\Omega _{x_{L}}\diagdown \{FE_{e_{1}}\},Ad(g_{i})(\Omega
_{x_{i}}):i=2,...,M,F=R,I\right) \\
&=&sp\left( FE_{e_{1}},FE_{e_{1}+e_{k}}:k\geq 2,FE_{e_{1}-e_{j}},j\in
J_{0},F=R,I\right) .
\end{eqnarray*}

If $m^{\prime }<m,$ then $J_{0}=\{2,...,n\}$ and we complete the proof by
taking $M=RE_{e_{1}}$ and appealing to the general strategy.

Otherwise, the number of indices in $\{2,...,n\}\diagdown J_{0}$ (we call
these the `missing' indices) is 
\begin{equation*}
N:=n-1-\left(
\sum_{i=1}^{m}(n-v_{i})+\sum_{i=m+1}^{L-1}(n-u_{i})+n-s_{L}+m-1\right) .
\end{equation*}%
The eligibility assumption (\ref{El2}) implies 
\begin{eqnarray*}
N &\leq
&n(1-L)+\sum_{i=1}^{m}2v_{i}+\sum_{i=m+1}^{L-1}2u_{i}+s_{L}-%
\sum_{i=1}^{m}v_{i}-\sum_{i=m+1}^{L-1}u_{i}-m \\
&\leq &n(L-1)-\sum_{i=1}^{m}v_{i}-\sum_{i=m+1}^{L-1}u_{i},
\end{eqnarray*}%
which coincides with the cardinality of $\tbigcup_{i=1}^{L-1}J_{i}$. As $%
\tbigcup_{i=1}^{L}J_{i}\subseteq J_{0},$ we deduce that $N\leq (n-1)/2$.
There is no loss of generality in assuming the missing indices are $%
\{n-N+1,...,n\}$ (in other words, $J_{0}=\{2,...,n-N\}$), and the indices in 
$\tbigcup_{i=1}^{L-1}J_{i}$ include $\{2,...,N+1\}$. Note that $N+1<n-N+1$.

These observations imply that 
\begin{equation*}
H_{0}:=\sum_{k=2}^{N+1}RE_{e_{k}+e_{n-N-1+k}}\in \mathfrak{g}_{n-1}
\end{equation*}%
is well defined. Furthermore, whenever $j\in \{2,...,N+1\},$ then%
\begin{equation*}
ad(H_{0})(FE_{e_{1}+e_{j}})=a_{F}FE_{e_{1}-e_{n-N-1+j}}
\end{equation*}%
where $a_{F}\neq 0$.

We will now appeal to Lemma \ref{elem}, where we take $\Xi
_{i}=Ad(g_{i})\Omega _{x_{i}}$ for $2\leq i\leq m$, $\Xi _{j}=\Omega _{x_{j}}
$ for $j=1$ and $m<j\leq L$, $\Omega _{0}=\{RE_{e_{1}},IE_{e_{1}}\}$ and 
\begin{equation*}
\Omega _{\ast }=\{FE_{e_{1}+e_{k}}:k=2,...,N+1,F=R,I\}.
\end{equation*}%
As the root vectors $FE_{e_{1}+e_{j}},$ for each choice of $j\in
\tbigcup_{i=1}^{L-1}J_{i},$ occur in both the set%
\begin{equation*}
\Omega _{x_{1}}\tbigcup \tbigcup\limits_{k=m+1}^{L-1}\Omega _{x_{k}}\tbigcup
\tbigcup\limits_{k=2}^{m}Ad(g_{k})\Omega _{x_{k}}
\end{equation*}%
and the set $\Omega _{x_{L}}$, we see that 
\begin{equation*}
\Omega _{\ast }\subseteq \left( \Xi _{L}\tbigcap
\tbigcup\limits_{i=1}^{L-1}\Xi _{i}\right) \diagdown \Omega _{0}.
\end{equation*}%
Since $J_{0}=\{2,...,n-N\}$ and%
\begin{equation*}
ad(H_{0})(\Omega _{\ast })=\{FE_{e_{1}-e_{k}}:k\geq n-N+1,F=R,I\},
\end{equation*}%
follows that 
\begin{eqnarray*}
&&sp\left( ad(H_{0})\Omega _{\ast },\Xi _{L}\diagdown \Omega
_{0},\tbigcup\limits_{i=1}^{L-1}\Xi _{i}\diagdown \Omega _{\ast }\right)  \\
&=&sp\left(
FE_{e_{1}},FE_{e_{1}+e_{i}},FE_{e_{1-}e_{j}},FE_{e_{1}-e_{k}}:i\geq 2,j\in
J_{0},k\geq n-N+1,F=R,I\right)  \\
&=&sp\Omega 
\end{eqnarray*}%
Thus Lemma \ref{elem} implies property (i) of the general strategy is
satisfied with the $g_{i}$ there replaced by $\exp tH_{0}\cdot g_{i}$ for $%
i=2,...,m$. We now complete the proof in the usual way with $M=RE_{e_{1}}$.

\smallskip

Lastly, suppose $m=0$ and put $X_{i}=X_{x_{i}}$. Then $X_{1},...,X_{L-1}$
are of dominant type $B$ in $\mathfrak{g}_{n}$ and $X_{L}$ is of dominant
type $SU(n)$. The eligibility condition for $(x_{1},...,x_{L})$ implies $%
(X_{x_{1}},...,X_{x_{L}})$ satisfies the eligibility requirement in the Lie
algebra setting (\ref{LieElig}) and since $\Omega _{X_{i}}^{\mathfrak{g}%
_{n}}=\Omega _{x_{i}}$, the proof given in \cite{CJM} again establishes
property (i) of the general strategy for $(x_{1},...,x_{L}),$ with a
suitable choice of $\Omega _{0}$. With an appropriate choice of $M$, as in 
\cite{CJM}, we complete the proof in the usual manner.
\end{proof}

\bigskip

\begin{proof}
\lbrack of Theorem \ref{main}] The necessity of eligibility was shown in
Proposition \ref{nec}, thus we only need prove its sufficiency. We proceed
by induction on the rank $n$ of $G_{n}$. Lemma \ref{ind} establishes the
base case, thus we may assume that $n\geq 3$ and that all eligible,
non-central tuples in $G_{n-1}$ are absolutely continuous.

Let $(x_{1},...,x_{L})$ be an eligible tuple with each $x_{j}\in T_{n}$. If
two or more $x_{i}$ are of dominant type $S$ or $BD,$ then according to
Lemma \ref{1} we even have $\mu _{x_{1}}\ast \cdot \cdot \cdot \ast \mu
_{x_{L}}\in L^{2},$ so the convolution is clearly absolutely continuous.

Thus we may assume at most one $x_{j}$ is of dominant type either $S$ or $BD$%
.

First, assume there are none. Suppose $x_{1},...,x_{m}$ are each of dominant
type $D$ and $x_{m+1},...,x_{L}$ are each of dominant type $B,$ with $x_{i}$
of type $B_{u_{i}}\times D_{v_{i}}\times SU(s_{i,1})\times \cdot \cdot \cdot
\times SU(s_{i},_{t_{i}})$. Simple calculation shows that if $m$ is odd,
then eligibility means 
\begin{equation}
\sum_{i=1}^{m}2v_{i}+\sum_{i=m+1}^{L}2u_{i}\leq 2n(L-1)+m-1\text{.}
\label{Elig1}
\end{equation}

Consider $y_{1}$ of type $B_{v_{1}}\times D_{u_{1}}\times SU(s_{1,1})\times
\cdot \cdot \cdot \times SU(s_{1,t_{1}})$. We have $v_{1}>u_{1}$ and $%
2v_{1}>s_{1},$ hence $y_{1}$ is of dominant type $B$. Furthermore, by taking
a Weyl conjugate, if necessary, we can assume $y_{1}$ has the same
annihilating roots of the form $e_{i}\pm e_{j}$ as $x_{1}$, and more
annihilating roots of the form $e_{k}$, so $\mathcal{N}_{x_{1}}\supseteq 
\mathcal{N}_{y_{1}}$. Consequently, Proposition \ref{charabscont} implies
that if $(y_{1},x_{2},...,x_{L})$ is absolutely continuous, so is $%
(x_{1},x_{2},...,x_{L})$. Moreover, as the tuple $(y_{1},x_{2},...,x_{L})$
has $m-1$ (an even number) of dominant type $D$ terms, one can check that
the eligibility requirement for this tuple coincides with (\ref{Elig1}), the
eligibility requirement for $(x_{1},x_{2},...,x_{L})$. This shows there is
no loss of generality in assuming that if the tuple admits only terms of
dominant type either $B$ or $D,$ then we can assume the number of terms of
dominant type $D$ is even.

Next, suppose $x_{1},...,x_{m}$ is of dominant type $D,$ $%
x_{m+1},...,x_{L-1} $ is of dominant type $B$ and $x_{L}$ is of dominant
type $BD$. Using the same notation as above, if $m$ is even, eligibility
requires 
\begin{equation*}
\sum_{i=1}^{m}2v_{i}+2v_{L}+\sum_{i=m+1}^{L-1}2u_{i}\leq 2n(L-1)+m.
\end{equation*}%
If we simply re-name $x_{L}$ as dominant type $D,$ then we are in the
situation of an odd number of dominant type $D$ (namely, $m+1$), and the
eligibility requirement for the case of a tuple with $m+1$ terms of dominant
type $D$ and $L-m-1$ terms of dominant type $B$ is satisfied.

Similarly, if $m$ is odd, eligibility requires 
\begin{equation*}
\sum_{i=1}^{m}2v_{i}+\sum_{i=m+1}^{L}2u_{i}\leq 2n(L-1)+m-1,
\end{equation*}%
and this is precisely the eligibility requirement if we, instead, view $%
x_{L} $ as dominant type $B$. Thus the case of one $x_{j}$ of dominant type $%
BD$ can be reduced to the case where all $x_{j}$ are of dominant type $B$ or 
$D$.

These observations reduce the problem to establishing the absolute
continuity of eligible tuples satisfying either the hypothesis of Lemma \ref%
{MainLemma}(a) or (b), and hence that Lemma completes the proof.
\end{proof}

\subsection{Applications}

Here are some consequences of our characterization theorem.

\begin{corollary}
(a) The product of conjugacy classes, $\tprod_{i=1}^{L}C_{x_{i}},$ has
positive measure (or non-empty interior) if and only if $(x_{1},...,x_{L})$
is eligible.

(b) If $x_{j}$ and $y_{j}$ are of the same dominant type and $%
S_{x_{j}}=S_{y_{j}}$ for all $j$, then $\mu _{x_{1}}\ast \cdot \cdot \cdot
\ast \mu _{x_{L}}$ is absolutely continuous if and only if $\mu _{y_{1}}\ast
\cdot \cdot \cdot \ast \mu _{y_{L}}$ is absolutely continuous.
\end{corollary}

\begin{proof}
These statements follow from our Theorem \ref{main}, together with
Proposition \ref{charabscont} in the case of (a) and the observation that $%
(x_{1},...,x_{L})$ is eligible if and only if $(y_{1},...,y_{L})$ is
eligible in (b).
\end{proof}

\begin{example}
Suppose each $x_{j}$ is of dominant type $D$ and $S_{x_{j}}=2v_{j}$. Take $%
y_{j}$ of type $D_{v_{j}}$. Then $(x_{1},...,x_{L})$ is absolutely
continuous if and only if $(y_{1},...,y_{L})$ is absolutely continuous. This
is despite the fact that $\Phi _{y_{j}}$ is a proper subset of $\Phi
_{x_{j}} $ (unless $x_{j}$ is conjugate to $y_{j}),$ so $\mathcal{N}%
_{x_{j}}\subsetneq N_{y_{j}}$ and $\dim C_{x_{j}}<\dim C_{y_{j}}$.
\end{example}

Since it is always the case that $\Phi _{X_{x}}^{\mathfrak{g}}\subseteq \Phi
_{x},$ if $(x_{1},...,x_{L})$ is absolutely continuous in $G,$ then $%
(X_{x_{1}},...,X_{x_{L}})$ is absolutely continuous in $\mathfrak{g}$. The
converse is true when the Lie group $G$ is Lie type $A_{n}$, but not when $G$
is Lie type $B_{n}$. For example, if $x$ is type $D_{n}$ in a Lie group of
type $B_{n}$, then $X_{x}$ is type $SU(n)$. Hence $(X_{x},X_{x})$ is
absolutely continuous for all choices of $n$. But $(x,x)$ is never eligible,
not even when $n=2$. Indeed, it was shown in \cite{GAFA} that if $x$ is type 
$D_{n}$ in $B_{n},$ then $\mu _{x}^{2n}$ is absolutely continuous and $\mu
_{x}^{2n-1}$ is purely singular. In fact, more can be said.

\begin{corollary}
Suppose $x_{j}\in G_{n}\diagdown Z(G_{n})$ for $j=1,...,L$.

(a) Suppose $L=2n-1$. The convolution $\mu _{x_{1}}\ast \cdot \cdot \cdot
\ast \mu _{x_{L}}$ is not absolutely continuous with respect to Haar measure
on $G_{n}$ if and only if all $x_{j}$ are type $D_{n}$.

(b) If $L=2n,$ then $\mu _{x_{1}}\ast \cdot \cdot \cdot \ast \mu _{x_{L}}$
is absolutely continuous and $\tprod_{i=1}^{2n}C_{x_{i}}$ has non-empty
interior.
\end{corollary}

\begin{proof}
(a) One can easily check that if $x\notin Z(G_{n})$ is not type $D_{n},$
then $S_{x}\leq 2n-1,$ with equality only if $x$ is type $B_{n-1},$ and that
for any $x\notin Z(G_{n})$, $S_{x}\leq 2n,$ with equality only if $x$ is
type $D_{n}$.

Suppose $x_{1}$ is not type $D_{n}$. If either $S_{x_{1}}<2n-1$ or $%
S_{x_{j}}<2n$ for some $j>1,$ then 
\begin{equation*}
\sum_{i=1}^{2n-1}S_{x_{i}}\leq (2n-1)+(2n-2)(2n)-1=(2n+1)(2n-2).
\end{equation*}%
Thus $(x_{1},...,x_{2n-1})$ is eligible and therefore $\mu _{x_{1}}\ast
\cdot \cdot \cdot \ast \mu _{x_{2n-1}}$ is absolutely continuous.

So assume $S_{x_{1}}=2n-1$ and all $S_{x_{j}}=2n$ for $j>1$. That means $%
x_{1}$ is of dominant type $B$ and $x_{2},...,x_{2n-1}$ are of dominant type 
$D$. Consequently, there are an even number of dominant type $D$. That means 
$(x_{1},...,x_{2n-1})$ is eligible provided $\sum_{i=1}^{2n-1}S_{x_{i}}\leq
(2n+1)(2n-2)+1$, and that is obviously true.

(b) One can similarly check that any $2n$-tuple of elements $x_{j}\notin
Z(G_{n})$ in $G_{n}$ is eligible.
\end{proof}

\begin{remark}
We note that (b) also follows from the fact that for all $x\notin Z(G_{n})$, 
$\mu _{x}^{2n}\in L^{2},$ as shown in \cite{GAFA}.
\end{remark}

\begin{corollary}
If $(x_{1},...,x_{L})$ is eligible, then there are matrices $g_{i}$ such
that $\tprod_{i=1}^{L}g_{i}x_{i}g_{i}^{-1}$ has distinct eigenvalues.
\end{corollary}

\begin{proof}
This is due to the fact that a set in $G$ with non-empty interior must
contain elements with distinct eigenvalues.
\end{proof}

It is an open problem to characterize the tuples $(x_{1},...,x_{L})$ such
that $\tprod_{i=1}^{2n}C_{x_{i}}$ admits an element with distinct
eigenvalues.

\subsection{Types $C_{n}$ and $D_{n}$}

We conclude with some comments on extending this characterization to the
other classical Lie algebras, types $C_{n}$ and $D_{n}$. Every element of
the torus of one of these Lie groups is conjugate to an element of the form

\begin{equation}
x=(\underbrace{1,...,1}_{u},\underbrace{-1,...,-1}_{v},\underbrace{%
e^{ia_{1}},...,e^{ia_{1}}}_{s_{1}},...,\underbrace{e^{(\pm
)ia_{m}},...,e^{(\pm )ia_{m}}}_{s_{m}})  \label{CD}
\end{equation}%
where the $0<a_{j}<\pi $ are distinct, $u\geq v,$ $s_{1}=\max $ $s_{j}$ and $%
u+v+s_{1}+\cdot \cdot \cdot +s_{m}=n$. The minus sign is only needed for $%
D_{n}$. Such an $x$ will be said to be of type 
\begin{equation*}
C_{u}\times C_{v}\times SU(s_{1})\times \cdot \cdot \cdot \times SU(s_{m})%
\text{ or }D_{u}\times D_{v}\times SU(s_{1})\times \cdot \cdot \cdot \times
SU(s_{m})\text{ }
\end{equation*}
as these are the Lie types of their sets of annihilating roots.

We will say $x$ is of \textbf{dominant type }$C$ (or $D)$ if $2u\geq s_{1}$
and then set $S_{x}=2u,$ and say $x$ is of \textbf{dominant type} $S$
otherwise and set $S_{x}=s_{1}$. We will say that $(x_{1},...,x_{L})$ is 
\textbf{eligible} if%
\begin{equation*}
\sum_{i=1}^{L}S_{x_{i}}\leq 2n(L-1).
\end{equation*}%
One can again show that eligibility is a necessary condition for absolute
continuity.

We will say the pair $(x,y)$ in $C_{n}$ or $D_{n}$ is \textbf{exceptional}
if:

(i) $x$ is type $C_{n/2}\times C_{n/2}$ (or $D_{n/2}\times D_{n/2}$) ($n$
even) and $y$ is the same type as $x$ or type $SU(n)$;

(ii) $x$ is type $C_{n/2}\times C_{n/2}$ and $y$ is type $C_{n/2}\times
C_{n/2-1}$ ($n$ even) or $x$ is type $C_{(n+1)/2}\times C_{(n-1)/2}$ and $y$
is type $C_{(n-1)/2}\times C_{(n-1)/2}$ ($n$ odd);

(iii) (in $D_{n}$ only) $x$ is type $SU(n)$ and $y$ is type $SU(n)$ or $%
SU(n-1).$

The pairs of (i) are eligible, but their reductions (defined in a similar
manner as in $B_{n}$) are not. In fact, it can be seen from \cite[Thm. 9.1]%
{Adv} that the pair of type $(C_{n/2}\times C_{n/2},C_{n/2}\times C_{n/2})$
is not absolutely continuous (although it is eligible). The pairs of (ii)
are entwined under the reduction process and it is unclear whether their
base case, $(C_{2}\times C_{1},C_{1}\times C_{1})$ in $C_{3},$ is an
absolutely continuous pair. The pairs in (iii) coincide with exceptional
pairs from the Lie algebra problem, whose absolute continuity was
undetermined in \cite{CJM}.

All non-exceptional, eligible tuples reduce to non-exceptional, eligible
tuples.

\begin{conjecture}
Suppose the non-central elements $x_{j}$ belong to the torus of $C_{n}$ or $%
D_{n}$ for $n\geq 5$ and that $(x_{1},...,x_{L})$ is not exceptional. Then $%
(x_{1},...,x_{L})$ is absolutely continuous if and only if it is eligible.
\end{conjecture}

It is natural to appeal to an induction argument on the rank of the Lie
group, as we did for type $B_{n}$. But for the Lie groups of types $C_{n}$
and $D_{n},$ the induction step will be much easier than it was for $B_{n}$
as it will follow from the arguments given for the proof in the Lie algebras
setting in \cite{CJM}. This is essentially because the eligibility condition
is the same for $(x_{1},...,x_{L})$ and $(X_{x_{1}},...,X_{x_{L}}),$ and
even though it need not be true that $\Phi _{X_{x}}^{\mathfrak{g}_{n}}=\Phi
_{x}$ for $\mathfrak{g}_{n}$ the Lie algebra of type $C_{n}$ or $D_{n},$ it
is true that $\mathcal{N}_{X_{x}}^{\mathfrak{g}_{n}}\diagdown \mathcal{N}%
_{X_{x^{\prime }}}^{\mathfrak{g}_{n-1}}=\mathcal{N}_{x}\diagdown \mathcal{N}%
_{x^{\prime }}$.

However, establishing the base cases will be onerous with our current
techniques. For $C_{n}$, this task is complicated by the fact that when $x$
is type $C_{2}\times C_{1}$ in the group $C_{3},$ then $\mu _{x}^{3}$ is not
absolutely continuous. In $D_{4}$ we have the same issue with $x$ of type $%
SU(4)$. In fact, for all $D_{n}$ and $x$ of type $SU(n),$ $\mu _{x}^{2}$ is
not absolutely continuous and that means we cannot immediately reduce to the
case of at most one $x_{j}$ of type $S,$ as we did for $B_{n}$. This latter
problem arose in the Lie algebra setting, as well, and a modification of the
approach taken there will work here.

For $C_{3},$ $C_{4}$ and $D_{4},$ it appears there will be other exceptional
tuples and that new methods will be needed to determine which of these, as
well as which of the exceptional tuples listed above, are absolutely
continuous.

\end{document}